\documentclass[11pt]{amsart}

\usepackage[utf8]{inputenc}
\usepackage[margin=1in]{geometry}
\usepackage[titletoc,title]{appendix}
\usepackage{pifont}
\usepackage{yfonts}
\RequirePackage[dvipsnames,usenames]{xcolor}
\usepackage[colorlinks=true,linkcolor=blue]{hyperref}%
\hypersetup{
    colorlinks=true,
    linkcolor=RoyalBlue,
    filecolor=magenta,      
    urlcolor=RoyalBlue,
    pdfpagemode=FullScreen,
    citecolor=BurntOrange
}

\usepackage{amsmath,amsfonts,amssymb,mathtools}
\usepackage[alphabetic,lite]{amsrefs}
\usepackage{enumitem}
\usepackage{graphicx,float}

\usepackage[ruled,vlined]{algorithm2e}
\usepackage{algorithmic}
\usepackage{amsthm}

\usepackage{ytableau}
\usepackage{comment}
\usepackage{bigints}
\usepackage{setspace}
\numberwithin{equation}{section}

\newtheorem{theorem}{Theorem}[section]
\newtheorem{defn}[theorem]{Definition}
\newtheorem{corollary}[theorem]{Corollary}
\newtheorem{lemma}[theorem]{Lemma}
\newtheorem{prop}[theorem]{Proposition}
\newtheorem{question}[theorem]{Question}
\newtheorem{remark}[theorem]{Remark}

\newcommand{\frakm}{\mathfrak{m}}
\newcommand{\frakn}{\mathfrak{n}}

\newcommand{\CC}{\mathbb{C}}
\newcommand{\ZZ}{\mathbb{Z}}

\newcommand{\CJ}{\mathcal{J}}
\newcommand{\kk}{\Bbbk}
\newcommand{\Fe}{F^e_*}

\DeclareMathOperator{\Hom}{Hom}

\DeclareMathOperator{\Sup}{sup}
\DeclareMathOperator{\Min}{min}

\DeclareMathOperator{\Lct}{lct}
\DeclareMathOperator{\Lce}{lce}
\DeclareMathOperator{\Fpt}{fpt}
\DeclareMathOperator{\Ht}{ht}
\DeclareMathOperator{\Char}{char}

\makeatletter
\@namedef{subjclassname@2020}{%
  \textup{2020} Mathematics Subject Classification}
\makeatother

\title{Singularities of Generic Linkage via Frobenius Powers}

\author{Jiamin Li}
\address{Department of Mathematics, Statistics, and Computer Science, University of Illinois at Chicago,
Chicago, IL 60607}
\email{jli283@uic.edu}
\subjclass[2020]{}

\begin{document}

\maketitle

\begin{abstract}
	Let $I$ be an equidimensional ideal of a ring polynomial $R$ over $\CC$ and let $J$ be its generic linkage.  We prove that there is a uniform bound of the difference between the F-pure thresholds of $I_p$ and $J_p$ via the generalized Frobenius powers of ideals. This provides evidence that the F-pure threshold of an equidimensional ideal $I$ is less than that of its generic linkage. As a corollary we recover a result on log canonical thresholds of generic linkage by Niu.
\end{abstract}

%%%%%%%%%%%%%%%%%%%%%%%%%%%%%%%%%%%%%%%%%%%%%%%%%%%%
\section{Introduction}
In this paper we will work with polynomial rings of the form $R=\kk[x_1, \cdots, x_n]$ over a perfect field $\kk$. First we will recall the definition of generic linkage of an ideal. Let $I = (f_1,\cdots,f_r)$ be an equidimensional ideal of $R$ with $\Ht(I) = c$. Define the polynomial ring $S=R[u_{ij}]_{c \times r}$ and define the ideal $L$ generated by the elements that form a regular sequence $L := (g_1,\cdots,g_c)$, where each $g_i$ is defined by
$$g_i = f_1 u_{i1} + \cdots + f_r u_{ir}.$$
Moreover we define the colon ideal $J := (L:I)$. Then $J$ is called the generic linkage of $I$ in $S$.

The study of generic linkage has a long history, some of the pioneer work are \cite{HunekeUlrich87}, \cite{HunekeUlrich88}, \cite{ChardinUlrich}. As pointed out in \cite{MaPageRG+} and \cite{Niu}, there is still much to explore in the singularity theory of generic linkage. Recently the singularities of linkage has been studied in \cite{Niu}, where the author showed that when passing to the generic linkage, the singularities of the pairs would not get worse. As mentioned in the introduction of \cite{Niu}, the singularity of generic linkage has been studied by Chardin and Ulrich, see \cite{ChardinUlrich}. There they proved that under the assumption that $I$ defines a complete intersection variety $X$ with rational (resp. F-rational) singularity, then its generic linkage has rational (resp. F-rational) singularity as well. A similar result in positive characteristic in regard to F-purity has been proved in \cite{MaPageRG+} as well with the assumption that the definig ideal reduces to a complete intersection. We recall their results below.

\begin{theorem}\label{main_thm_char_p}\cite{MaPageRG+}*{Corollary 4.4}
	Let $J$ be a generic linkage of an equidimensional ideal $I$ in $R=\kk[x_1,...,x_n]$ of height $c$ where $\kk$ is a perfect field of positive characteristic. Then 
	\begin{enumerate}
		\item If $I$ has a reduction generated by $r$ elements, then $\frac{c}{r}\Fpt_R(I) \leq \Fpt_S(J)$.
		\item If $I$ has a reduction generated by $c$ elements (e.g. $I$ is a complete intersection), then $\Fpt_R(I) \leq \Fpt_S(J)$. In particular, if $(R,I^c)$ is F-pure, then $(S,J^c)$ is F-pure as well.
		\item $\frac{c}{n}\Fpt_R(I) \leq \Fpt_S(J)$.
	\end{enumerate}
\end{theorem}

It is natural to ask the following question:
\begin{question}\label{ques}
	Without assuming the ideal $I$ to have a complete intersection reduction, do we still have $\Fpt_R(I) \leq \Fpt_S(J)$?
\end{question}

We believe that the answer to this question is positive and this note provides evidence for this inference. To be more precise, our main result implies that there is a uniform bound of the difference between $\Fpt_R(I_p)$ and $\Fpt_S(J_p)$ where $I_p$ and $J_p$ are reduction modulo p from the corresponding ideals in zero characteristic for $p$ large enough. A key intermediate result in positive characteristic is the following.
\begin{theorem}[Proposition \ref{lce_linkage}]\label{main_prop}
	Let $I$ be an equidimensional ideal with height $c$ in $\kk[x_1,...,x_n]$ where $\kk$ is perfect of positive characteristic and $J$ its generic linkage. Then 
	$$\Lce_R(I^c) = \Lce_S(L^c) \leq \Lce_S(J^c).$$
\end{theorem}

Using the above result we can deduce: 

\begin{theorem}[Main Theorem]\label{main_thm}
Let $J$ be the generic linkage of an equidimensional ideal $I$ in $R=\CC[x_1,...,x_n]$. There exists a prime $P$ such that for any $p > P$, there exists a constant $C_1$ independent of $p$ such that
\begin{align}
	\Fpt_S(L_p) \leq \Fpt_R(I_p) \leq \Fpt_S(L_p) + \dfrac{C_1}{p} \leq \Fpt_S(J_p) + \dfrac{C_1}{p}.
\end{align}

\end{theorem}
Though the above theorem is weaker than the positive answer of Question \ref{ques}, as a corollary we can recover the result on log canonical thresholds of generic linkage in \cite{Niu}. Note that in \cite{Niu}, the author also proved that $\Lct_R(I) = \Lct_S(L)$ in order to prove $\Lct_R(I) \leq \Lct_S(J)$.
\begin{corollary}\label{main_thm_niu}
	Let $J$ be the generic linkage of an equidimensional ideal $I$ in $R=\CC[x_1,...,x_n]$. Then we have
	$$\Lct_R(I) = \Lct_S(L) \leq \Lct_S(J).$$
	Therefore, if the pair $(R, I^c)$ is log canonical, then $(S, J^c)$ is log canonical as well.
\end{corollary}

Our tool here is the generalized Frobenius powers ideals, which was introduced in \cite{HernandezTexeiraWitt}. To the best of our knowledge, this note can be considered as the first application in singularity theory of the generalized Frobenius powers of ideals.

\section{Preliminaries on F-singularity in positive characteristic}
We begin this section by recalling the notions of multiplier ideals and log canonical singularity in zero characteristic for a nonsingular algebraic variety $X$. For details see \cite{Lazarsfeld}*{Ch 9}. 
\begin{defn}\label{multiplier}
	Let $I$ be a nonzero ideal sheaf of $\mathcal{O}_X$. Let $n > 0$ be an integer, the multiplier ideal of $I^n$, denoted by $\CJ(I^n)$, is defined by
	\begin{align*}
		\CJ(I^n) = \phi_*\mathcal{O}_{X'}(K_{X'/X}-[n.F]),
	\end{align*}
	where $-F$ is the divisor associated to $I$, $\phi:X' \rightarrow X$ is the log resolution of $X$. 
\end{defn}

We can measure the singularity of a pair $(\mathcal{O}_X, I)$ with the log canonical threshold.

\begin{defn}\label{lct}
	Let $I$ be an ideal sheaf of $\mathcal{O}_X$. The log canonical threshold of $I$ at a point $x \in X$ is defined to be
	\begin{align*}
		\Lct(I;x) := \operatorname{inf}\{n \in \mathbb{R}; \CJ(I^n)_x \in \frakm_x\}.
	\end{align*}
	and 
	\begin{align*}
		\Lct_X(I) := \operatorname{inf}\{x \in X; \Lct(I;x)\}.
	\end{align*}
	We sometimes write $\Lct(I)$ if the underlying space is clear. Furthermore, for an ideal $I \subseteq R$ we can define $\Lct_R(I)$ similarly.
\end{defn}
Roughly speaking, the higher the log canonical threshold is, the less singular the pair is, as can be seen in Corollary \ref{main_thm_niu}. The same goes for its analogue in positive characteristic, which we will recall below.

Let $R$ be a Noetherian ring with positive characteristic $p$ and denote the set of elements that are not in any minimal prime of $R$ by $R^\circ$. We recall the definition of several notions of F-singularities. Let $F^e : R \rightarrow R$ be the $e$-th Frobenius map that sends $r$ to $r^{p^e}$, and we denote the image of $R$ to be $\Fe R$.
\begin{defn}\label{defn_f_pure}(see \cite{TakagiWanatabe}*{Definition 1.3})
	Let $I$ be an ideal of a reduced $F$-finite ring $R$ such that $I \cap R^\circ \neq \varnothing$. Let $t\in \mathbb{R}_{\geq 0}$.
	\begin{enumerate}	
		\item The pair $(R,I^t)$ is said to be F-pure if for $q=p^e$ large enough, there exists an element $a \in I^{\lfloor t(q-1) \rfloor}$ such that the inclusion $\Fe aR \hookrightarrow \Fe R$ splits over $R$.
		\item The pair $(R,I^t)$ is said to be strongly F-pure if there exists $q=p^e$ large enough, there exists $a \in I^{\lceil tq \rceil}$ such that the inclusion $\Fe aR \hookrightarrow \Fe R$ splits over $R$.
		\item The pair $(R,I^t)$ is said to be strongly F-regular if there exists $q=p^e$ large enough such that for every $c \in R^\circ$, there exists $a \in I^{\lceil tq \rceil}$ such that the inclusion $\Fe(ca) R \hookrightarrow \Fe R$ splits over $R$.
	\end{enumerate}
\end{defn}

We will recall definition of the generalized test ideal, which is an analogue of multiplier ideal in positive characteristic.

\begin{defn}\label{defn_test_ideal}
	Suppose $R$ is an F-finite Noetherian integral domain with $\Char(R) = p > 0$. Then we define the test ideal of $(R,I^t)$, denoted by $\tau(I^t)$, to be the unique smallest nonzero ideal $J$ of $R$ such that for all $\phi \in \Hom_R(\Fe R, R)$ and all $q=p^e$, $e>0$, $\phi(\Fe I^{\lceil t(q-1) \rceil}J) \subseteq J$.
\end{defn}

\begin{remark}\label{gen_test_ideal}
	When $R$ is regular F-finite, e.g. when $R$ is a polynomial ring over an F-finite field, the test ideal defined above has a simpler presentation. First we define $I^{[1/p^e]}$ to be the smallest ideal $J$ such that $I \subseteq J^{[p^e]}$. Then it was shown in \cite{BlickleMustataSmith} that the ideal
	\begin{align*}
		\bigcup_{e\geq 0} (I^{\lceil tp^e \rceil})^{[1/p^e]}
	\end{align*}
	coincides with $\tau(I^t)$. Moreover, since the ring is assumed to be Noetherian, it is easy to see that in this situation $\tau(I^t) = (I^{\lceil tp^e \rceil})^{[1/p^e]}$ for $e$ large enough.
\end{remark}

Similar to the situation in zero characterstic, we have the analog notion of the log canonical threshold.
\begin{defn}\label{defn_fpt}
	Let $R$ be a reduced F-finite F-pure ring of positive characteristic.  The F-pure threshold, denote by $\Fpt_R(I)$, is defined to be
	\begin{align*}	
		\Fpt_R(I) = \Sup\{t \in \mathbb{R}_{>0} : (R,I^t) \text{ is F-pure}\}
	\end{align*}
	We sometimes write $\Fpt(I)$ if the underlying ring is clear.
\end{defn}

\begin{remark}
	When $R$ is strongly F-regular,	we can rephrase the above definition of F-pure threshold to be
	\begin{align*}	
		\Fpt_R(I) = \Sup\{t \in \mathbb{R}_{>0} : (R,I^t) \text{ is strongly F-regular}\}.
	\end{align*}
	In addition, if $R$ is a local ring, then $(R,I^t)$ is strongly F-regular if and only if $\tau(I^t) = R$, see \cite{MaPageRG+}*{Remark 2.8}.
\end{remark}

\textbf{Reduction modulo p.} As mentioned before, the F-pure threshold measures the singularity of a pair, which can be considered as an analogue of log canonical threshold in zero characteristic. Using the reduction modulo $p$ method, we can make precise their connection, which we will briefly recall here and we refer the reader to \cite{MustataZhang}*{Ch 2} for more details. Let $R=\CC[x_1,\cdots,x_n]$ and let $I=(f_1,\cdots,f_r)$ be an ideal in $R$. Let $A$ be a finitely generated $\ZZ$-algebra that contains all the coefficients of $f_1,\cdots, f_r$. Let $\frakm_p$ be a maximal ideal of $A$ such that the residue field $\kk(p) := A/\frakm_p$ is of characteristic $p > 0$. Then there exists a finitely generated $A$-algebra, denoted by $\mathfrak{R}_A$, such that $\mathfrak{R}_A \otimes_A \CC \cong R/I$. Let $I_A$ be the ideal in $\mathfrak{A}:=A[x_1,\cdots,x_n]$ generated by $f_1,\cdots,f_r$ mentioned before, and let $I_p$ denote the image of $I_A$ in $\mathfrak{A} \otimes_A \kk(p)$. Then $I_p$ is called the reduction modulo $p$ of $I$ and $\operatorname{Spec}(\mathfrak{R}_A \otimes_A \kk(p))$ is called a characteristic $p$ model of $\operatorname{Spec}(R/I)$.

We then recall here the known relations between the log canonical threshold of an ideal $I$ and the F-pure threshold of its reduction modulo $p$.
\begin{theorem}\label{fpt<lct}\cite{HaraYoshida}*{Proposition 3.8}
	Using the notations above, after possibly replacing $A$ by its localization $A_a$ where $a\in A$ is nonzero, for any maximal ideal $\frakm_p$ of $A$ with $\Char(\kk(p)) = p$, we have $\Lct(I) \geq \Fpt(I_p)$.
\end{theorem}
A uniform bound of $\Lct(I) - \Fpt(I_p)$ was obatined by Musta\c{t}\u{a} and Zhang.
\begin{theorem}\cite{MustataZhang}*{Theorem A (i)}\label{diff_fpt_lct}
	Using the same notations as in Theorem \ref{fpt<lct}, there exists a constant $C > 0$ such that for any maximal ideal $\frakm_p$ of $A$ with $\Char(\kk(p))=p$, we have
	\begin{align*}
		\Lct(I) - \Fpt(I_p) < \dfrac{C}{p}
	\end{align*}
\end{theorem}

We also recall an important property of the F-pure thresholds of powers of ideals, which will play a key role in the proof of the main theorem.

\begin{lemma}\cite{TakagiWanatabe}*{Proposition 2.2 (2)}\label{fpt=cfpt}
	Let $R$ be a reduced F-finite F-pure ring of positive characteristic. Then for any ideal $I \subseteq R$ and any positive integer $n$ we have $\Fpt(I) = n\Fpt(I^n)$.
\end{lemma}

\section{Frobenius powers}
In this section we assume the ring $R$ is of positive characteristic. We recall the constructions of the Frobenius powers of ideals from \cite{HernandezTexeiraWitt}. For any $e \geq 0$, we define the $p^e$-Frobenius power of $I$ to be
\begin{align*}
	I^{[p^e]} = \{f^{p^e}: f \in I\}.
\end{align*}
Moreover, we define the $[1/p^e]$-Frobenius power of $I^{[1/p^e]}$ to be the smallest ideal $J$ such that $I \subseteq J^{[p^e]}$. Then we can generalize these definitions as follows.

\begin{defn}\label{real_frob_power}
Let $k \in \mathbb{N}$ and let $$k = k_0 + k_1 p + \cdots + k_n p^n$$ be the base $p$ expansion of $k$ for a prime $p$, then we define
$$I^{[k]} := I^{k_0}I^{k_1[p]}\cdots I^{k_n[p^n]}.$$

Moreover, let $e \geq 0$. Then we define the rational power $[k/p^e]$ of $I$ to be
$$I^{[k/p^e]} := (I^{[k]})^{[1/p^e]}.$$

Finally, for $t \in \mathbb{R}_{\geq 0}$ we define the $[t]$-real Frobenius power of $I$ to be
$$I^{[t]} = \bigcup_{k \geq 0} I^{[t_k]}$$ where $t_k \searrow t$. For Noetherian ring this means $I^{[t]} = I^{[t_k]}$ for $k$ large enough.

\begin{remark}\label{p_seq_t}
	As noted in \cite{HernandezTexeiraWitt}, in practice we usually take the sequence $\{t_k = \lceil tp^k \rceil / p^k\}$.
\end{remark}
\end{defn}

We recall the definition of the least critical exponent of an ideal, which serves as an analog of the F-pure thresholds.

\begin{defn}\label{defn_lce}
	The least critical exponent of an ideal $I$ is defined to be
	$$\Lce_R(I) = \Min\{t \in \mathbb{R}_{>0}; I^{[t]} \neq R\} = \Sup\{t \in \mathbb{R}_{>0}; I^{[t]} = R\}.$$
\end{defn}

Note that the minimum can be achieved for both the F-pure threshold and the least critical exponent of $I$ due to the right constancy of $\tau(I^t)$ and $I^{[t]}$, see 
\cite{BlickleMustataSmith}*{Proposition 2.14} and \cite{HernandezTexeiraWitt}*{Lemma 3.10}.

\begin{remark}\label{alg_closed}
	By \cite{BlickleMustataSmith}*{Remark 2.18} we see that if $\kk \subseteq \mathbb{K}$ is an extension of perfect field and we let $R=\kk[x_1,\cdots,x_n]$ and $R' = \mathbb{K}[x_1,\cdots,x_n]$, then we have $I^{[t]}R' = (IR')^{[t]}$. 
\end{remark}
We end this section by recording two important theorems relating the log canonical threshold, F-pure threshold and least critical exponent.

\begin{theorem}(See \cite{HernandezTexeiraWitt}*{Theorem C})\label{htwthm}
	Let $I$ be an ideal of a polynomial ring $R$ over $\CC$, then in general we have $\Lce(I_p) \leq \Fpt(I_p) \leq \Lct(I)$ for any prime $p$. If in addition $\Lct(I) \leq 1$, then we have 
	$$\lim_{p \rightarrow \infty}\Lce(I_p) = \lim_{p \rightarrow \infty}\Fpt(I_p) = \Lct(I).$$
\end{theorem}

\begin{theorem}(See \cite{HernandezTexeiraWitt}*{Proposition 4.16}\label{fpt<lce+}
	Suppose $I$ is a nonzero proper ideal of a polynomial ring $R$ over a F-finite field of positive characteristic. Then
	\begin{enumerate}
		\item If $0 \neq f \in I$, then we have $$0 \leq \Fpt(f) = \Lce(f) \leq \Lce(I) \leq \Min\{1,\Fpt(I)\}.$$
		\item If $\Lce(I) \neq 1$ and if $I$ can be generated by $r$ elements, then $$\Fpt(I) \leq \Lce(I)+\dfrac{r-1}{p-1}.$$ 
	\end{enumerate}
\end{theorem}
\section{Proof of the main theorem}
In this section we will prove the main results of this paper. A crucial step is a relation between the least critical exponents of the ideal $I$ and its linkage $J$. Before we state the result we recall a very useful theorem of the multiplier ideal of powers. 
\begin{theorem}(Variant of Skoda Theorem, see \cite{Lazarsfeld}*{9.6.37})\label{skoda}
	Suppose that $I$ is an ideal sheaf of $\mathcal{O}_X$ in zero characteristic where $X$ is a nonsingular algebraic variety. Suppose that all the associated subvarieties of $I$ has codimension less than or equal to $n$, then 
	\begin{align*}
		\CJ (I^m) \subseteq I
	\end{align*}
	for $m \geq n$.
\end{theorem}
\begin{remark}
Note that we also have a similar property for the test ideal in positive characteristic, see \cite{MaPageRG+}*{2.13.1}.
\end{remark}

With the Skoda theorem we get the following lemma.
\begin{lemma}\label{lct<=1}
	Let $I$ be an equidimensional ideal with $\Ht(I) = c$ in a polynomial ring $R$ over $\CC$ and let $I_p$ be its reduction modulo $p$, then 
	\begin{align*}
		\Lce(I_p^c) \leq \Fpt(I_p^c) \leq \Lct(I^c) \leq 1.
	\end{align*}
\end{lemma}

\begin{proof}
	This follows immediately from the definition of $\Lct(I)$, Theorem \ref{htwthm} and Theorem \ref{skoda}.
\end{proof}

We will make use of the below elementary property of multinomial in the proof of the main theorem as well.

\begin{theorem}[Multinomial Theorem]\label{multinomial_thm}
	Let $k,m\in \mathbb{N}$. Then we can expand the $k$-th power of the sum of $m$ terms as 
	$$(y_1+ \cdots +y_m)^k = \sum_{\gamma_1+ \cdots +\gamma_m =k}\binom{k}{\gamma_1,...,\gamma_m}y_1^{\gamma_1}\cdots y_m^{\gamma_m},$$
	where $$\binom{k}{\gamma_1,\cdots ,\gamma_m} = \dfrac{k!}{\gamma_1! \cdots \gamma_m!}.$$
\end{theorem}

Next, we prove the crucial intermediate step (Theorem \ref{main_prop}).
\begin{prop}\label{lce_linkage}
	Assume $\kk$ is a perfect field of positive characteristic $p$. Let $I = (f_1,\cdots,f_r)$ be an ideal in $R = \kk[x_1,\cdots, x_n]$ of height $c$ and let $J$ be its generic linkage. Recall also that $L=(g_1,\cdots, g_c)$ where $g_i:= \sum u_{ij}f_j$.  Then we have
	$$\Lce_R(I^c) = \Lce_S(L^c) \leq \Lce_R(J^c).$$
\end{prop}
\begin{proof}
	By Remark \ref{alg_closed}, without loss of generality, we may assume $\kk$ is algebraically closed. Since $L^c \subseteq J^c$ we have $\Lce_S(L^c) \leq \Lce_S(J^c)$, it remains to prove the first equality in the statement. We first observe that $\Lce_S(I^c) = \Lce_R(I^c)$, and since we have $L^c \subseteq I^c$ in $S$, we have $\Lce_S(L^c) \leq \Lce_R(I^c)$.

	We will prove the equality by contradiction, so suppose $\Lce_S(L^c) < \Lce_R(I^c)$ and suppose that $\Lce_R(I^c) = t$, which means there exists $\epsilon>0$ such that $(L^c)^{[t-\epsilon]} \neq S$, 

	As noted in remark \ref{p_seq_t},  we can take $t_k = \lceil tp^k \rceil/p^k$ for $k$ large enough such that $(I^c)^{[t]}=(I^c)^{[t_k]}$ and $(L^c)^{[t]} = (L^c)^{[t_k]}$. We also let $\epsilon = 1/p^e$, in particular we let $e>k$. Since $t\leq 1$, we can write $t_k=1-a/p^k$ for some $a<p^k$. Then the base $p$ expansion of $a$ can be written as $$a=a_0 + pa_1 + \cdots+ p^la_l, \quad l<k<e, \quad a_i \leq p-1$$ let $k' = e-k$ so that we have the base $p$ expansion 
	\begin{align*}
		&p^e-p^{k'}a-1\\
			      &=(p-1)(p^{e-1}+p^{e-2}+ \cdots + p+1) - p^{k'}(a_0+a_1p+ \cdots +a_lp^l)\\
			      &=(p-1) + p(p-1) + \cdots + p^{k'}(p-1-a_0) + \cdots + p^{k'+l}(p-1-a_l)+ \cdots + p^{e-1}(p-1).
	\end{align*}

	Then by definition of the $p$-rational Frobenius powers we can write $(I^c)^{[t_k-\epsilon]}$ as  
	\begin{align}\label{exp_ideal_I}
		\begin{split}
			((I^c)^{[1-a/p^k-1/p^e]}) &= ((I^c)^{[p^e-p^{k'}a-1])})^{[1/p^e]}\\
						  &= (\prod^{e-1}_{i=0}(I^c)^{b_i[p^i]})^{[1/p^e]}, \quad b_i \leq p-1,
		\end{split}
	\end{align}
	where the $b_i$ correspond to the coefficients in the above base $p$ expansion. Similarly we can write 
	\begin{align}\label{exp_ideal_L}
		\begin{split}
			((L^c)^{[1-a/p^k-1/p^e]}) &= ((L^c)^{[p^e-p^{k'}a-1])})^{[1/p^e]}\\
						  &= (\prod^{e-1}_{i=0}(L^c)^{b_i[p^i]})^{[1/p^e]}, \quad b_i \leq p-1.
		\end{split}
	\end{align}
	By our assumption, there exists a maximal ideal $\frakm\subset S$ such that $(L^c)^{[t_k-1/p^e]} \subseteq \frakm$. Since $\kk$ is algebraically closed, we can assume $\frakm = (x_1 - \alpha_1,\cdots ,x_n - \alpha_n, u_{11} - \beta_{11},\cdots, u_{cr} - \beta_{cr})$, and let $\frakn = (x_1 -\alpha_1,\cdots, x_n - \alpha_n)$. Then we have
	\begin{align*}
		\prod^{e-1}_{i=0}(L^c)^{b_i[p^i]} \subseteq \frakm^{[p^e]}.
	\end{align*}
	Pick a generator $g_1^{w_{i1}}\cdots g_c^{w_{ic}}$ of each $(L^c)^{b_i[p^i]}$. By definition of Frobenius power we have $\sum_j w_{ij} = cb_ip^i$ and $p^i$ divides each $w_{ij}$. In particualr we can take $g_1^{b_ip^i}\cdots g_c^{b_ip^i} \in (L^c)^{b_i[p^i]}$.

	Recall that $g_j = u_{j1}f_1 + \cdots + u_{jr}f_r$, and so by Theorem \ref{multinomial_thm} we have the expansion
	\begin{align}\label{exp_gj}
		\begin{split}
		g_j^{b_ip^i} &= (u_{j1}f_1+\cdots+u_{jr}f_r)^{b_ip^i}\\
			     &= ((u_{j1}f_1)^{p^i} + \cdots + (u_{jr}f_r)^{p^i})^{b_i}\\
			     &= \sum_{\gamma_1+\cdots \gamma_r = b_i} \binom{b_i}{\gamma_1,\cdots,\gamma_r}(u_{j1}f_1)^{p^i\gamma_1}\cdots(u_{jr}f_r)^{p^i\gamma_r}
	\end{split}
	\end{align}
	Note that since $b_i \leq p-1$, the multinomial coefficient $\binom{b_i}{\gamma_1,\cdots,\gamma_r}$ is not zero modulo $p$.

	On the other hand, pick an arbitrary generator $f_i$ of $(I^c)^{b_i[p^i]}$, then we can expand $f_i$ as $$f_i = f_1^{v_{i1}}\cdots f_r^{v_{ir}}$$ where $\sum v_{ij} = cb_ip^i$ and $p^i$ divides each $v_{ij}$. Then by the calculation of (\ref{exp_gj}) it is easy to see that $f_i$ will appear in the expansion of $g_1^{b_ip^i}\cdots g_c^{b_ip^i}$, and so any generator $f$ of $(I^c)^{[p^e-p^{k'}a-1]}$ will appear in the expansion of the product $g:=\prod_i g_1^{w_{i1}}\cdots g_c^{w_{ic}}$, which is an element of $(L^c)^{[p^e-p^{k'}a-1]}$. Fix an $f$, we can write $$g = C.\prod u_{ij}^{s_{ij}} f + \text{ other terms},$$ where $C$ is the product of corresponding multinomial coefficients, which as mentioned above is not zero modulo $p$.

	Next, define the $R$-homomorphism $\pi:S \rightarrow S$ such that $\pi(\prod u_{ij}^{s_{ij}}) = \prod u_{ij}^{s_{ij}}$ and $\pi(\prod u_{ij}^{s'_{ij}})=0$ if $s'_{ij} \neq s_{ij}$. Then $\pi(g) = C.\prod u_{ij}^{s_{ij}}f$ and we have $\pi(g) \in \pi(\frakm^{[p^e]}) \subseteq \frakm^{[p^e]}$. But $s_{ij} \leq p^e-1$, so $\prod u_{ij}^{s_{ij}} \notin \frakm^{[p^e]}$, which means $f \in \frakm^{[p^e]}$. However that means $f \in \frakm^{[p^e]} \cap R = \frakn^{[p^e]}$, a contradiction. Therefore $(L^c)^{[t_k-1/p^e]} = S$, and therefore $\Lce_S(L^c) \geq \Lce_R(I^c)$, and thus $\Lce_R(I^c) = \Lce_S(L^c)$, as desired.

\end{proof}
With the above result, we can prove our main result now.
\begin{proof}[Proof of Theorem \ref{main_thm}]
	 Let $P$ be a prime number large enough such that for any $p > P$, $(f_i)_p$ is nonzero for all $1 \leq i \leq r$ and $I_p$ is equidimensional of height $c$. By construction each $(g_i)_p$ is nonzero as well.

	By Theorem \ref{fpt<lct} and Theorem \ref{diff_fpt_lct}, there exists a constant $\epsilon_1 > 0$ such that for any prime $p$ we have
	\begin{align*}
		\Fpt_R(I_p) \leq \Lct_R(I) \leq \Fpt_R(I_p) + \dfrac{\epsilon_1}{p}.\\
	\end{align*}
	By Lemma \ref{fpt=cfpt} the right hand side is equal to
	\begin{align}\label{cfpt+e1p}
			  c\Fpt_R(I_p^c) + \dfrac{\epsilon_1}{p}.
	\end{align}
	By Lemma \ref{lct<=1}, $\Lce(I_p^c) \leq \Fpt(I_p^c) \leq 1$. Suppose $I^c_p$ is generated by $k$ elements (note that $k$ is independent of $p$), then by Theorem \ref{fpt<lce+} we have the below inequality. Note that if $\Lce_R(I_p^c)=1$ then $\Fpt_R(I_p^c)=1$ so the inequality is trivially true.
	$$\Fpt_R(I_p^c) \leq \Lce_R(I_p^c) + \dfrac{k-1}{p-1},$$
	and therefore by Proposition \ref{lce_linkage} we have for any $p > P$,
	\begin{align*}
		 &(\ref{cfpt+e1p}) \leq c(\Lce_R(I_p^c) + \dfrac{k-1}{p-1}) + \dfrac{\epsilon_1}{p}\\
			  & \leq c(\Lce_S(L_p^c) + \dfrac{k-1}{p-1}) + \dfrac{\epsilon_1}{p}\\
			  & \leq c(\Fpt_S(L_p^c) + \dfrac{k-1}{p-1}) + \dfrac {\epsilon_1}{p}\\
			  & \stackrel{\text{assume }p > k}{\leq} c(\Fpt_S(L_p^c) + \dfrac{k}{p}) + \dfrac {\epsilon_1}{p}\\
			  & = \Fpt_S(L_p) + \dfrac{\epsilon_2}{p},
	\end{align*}
	Since $L_p \subseteq J_p$, we end up with
	$$\Fpt_R(I_p) \leq \Fpt_S(L_p) + \dfrac{\epsilon_2}{p} \leq \Fpt_S(J_p) + \dfrac{\epsilon_2}{p}.$$
	where $\epsilon_2$ is a constant independent of $p > \max\{k, P\}$. Since $L_p \subseteq I_p$ in $S$, we have $\Fpt_S(L_p) \leq \Fpt_S(I_p) = \Fpt_R(I_p)$, and so we get the desired inequalities.
\end{proof}

\begin{proof}[Proof of Corollary \ref{main_thm_niu}]
	Taking limit on both sides of the above inequalities and letting $p$ goes to infinity, by Theorem \ref{fpt<lct} we get 
	\begin{align*}
		\lim_{p \rightarrow \infty}\Fpt_S(L_p) \leq &\lim_{p \rightarrow \infty}\Fpt_R(I_p) \leq \lim_{p \rightarrow \infty}\Fpt_S(L_p) + \dfrac{\epsilon_2}{p} \leq \lim_{p \rightarrow \infty}\Fpt_S(J_p) + \dfrac{\epsilon_2}{p}\\
		&\implies \Lct_R(I) = \Lct_S(L) \leq \Lct_S(J).
	\end{align*}
	which finishes the proof.
\end{proof}

\section{Acknowledgement}
The author would like to thank her advisor Wenliang Zhang for suggesting this problem and useful discussions, and to Linquan Ma for helpful comments on the earlier draft of this paper. The author is partially supported by the NSF grant DMS 1752081.

%%%%%%%%%%%%%%%%%%%%%%%%%%%%%%%%%%%%%%%%%%%%%%%%%%%%%%%%%%%%%%%%%%%%%%%%


\begin{thebibliography}{HKM}
%%%%%%%%%%%%%%%%%%%%%%%%%%%%%%%%%%%%%%%%%%%%%%%%%%%%%%%%%%%%%%%%%%%%%%%%
	\bib{BlickleMustataSmith}{article}{
		title={Discreteness and rationality of F-thresholds},
		author={Blickle, Manuel}, author={Mustata, Mircea}, author={Smith, Karen E},
		journal={Michigan Mathematical Journal},
		volume={57},
		pages={43--61},
		year={2008},
		publisher={University of Michigan, Department of Mathematics}
	}

	\bib{ChardinUlrich}{article}{
		title={Liaison and Castelnuovo-Mumford regularity},
		author={Chardin, Marc}, author={Ulrich, Bernd},
		journal={American journal of mathematics},
		volume={124},
		number={6},
		pages={1103--1124},
		year={2002},
		publisher={Johns Hopkins University Press}
	}

	\bib{HunekeUlrich87}{article}{
		title={The structure of linkage},
		author={Huneke, Craig}, author={Ulrich, Bernd},
		journal={Annals of Mathematics},
		volume={126},
		number={2},
		pages={277--334},
		year={1987},
		publisher={JSTOR}
	}

	\bib{HunekeUlrich88}{article}{
		title={Algebraic linkage},
		author={Huneke, Craig}, author={Ulrich, Bernd},
		journal={Duke Math. J.},
		volume={56},
		pages={415--429},
		year={1988}
	}

	\bib{Lazarsfeld}{book}{
		AUTHOR = {Lazarsfeld, Robert},
		TITLE = {Positivity in algebraic geometry. {II}},
		SERIES = {Ergebnisse der Mathematik und ihrer Grenzgebiete. 3. Folge. A
			Series of Modern Surveys in Mathematics [Results in
			Mathematics and Related Areas. 3rd Series. A Series of Modern
		Surveys in Mathematics]},
		VOLUME = {49},
		PUBLISHER = {Springer-Verlag, Berlin},
		YEAR = {2004},
		PAGES = {xviii+385},
		ISBN = {3-540-22534-X},
		DOI = {10.1007/978-3-642-18808-4},
		URL = {https://doi-org.proxy.cc.uic.edu/10.1007/978-3-642-18808-4},
	}


	\bib{MaPageRG+}{article}{
		title={F-singularities under generic linkage},
		author={Ma, Linquan}, author={Page, Janet}, author={R.G., Rebecca}, author={Taylor, William}, author={Zhang, Wenliang},
		journal={Journal of Algebra},
		volume={505},
		pages={194--210},
		year={2018},
		publisher={Elsevier}
	}

	\bib{MustataZhang}{article}{
		title={Estimates for F-jumping numbers and bounds for Hartshorne--Speiser--Lyubeznik numbers},
		author={Musta{\c{t}}{\u{a}}, Mircea}, author={Zhang, Wenliang},
		journal={Nagoya Mathematical Journal},
		volume={210},
		pages={133--160},
		year={2013},
		publisher={Cambridge University Press}
	}

	\bib{Niu}{article}{
		title={Singularities of generic linkage of algebraic varieties},
		author={Niu, Wenbo},
		journal={American Journal of Mathematics},
		volume={136},
		number={6},
		pages={1665--1691},
		year={2014},
		publisher={Johns Hopkins University Press}
	}

	\bib{HaraYoshida}{article}{
		title={A generalization of tight closure and multiplier ideals},
		author={Hara, Nobuo}, author={Yoshida, Ken-Ichi},
		journal={Transactions of the American Mathematical Society},
		volume={355},
		number={8},
		pages={3143--3174},
		year={2003}
	}

	\bib{HernandezTexeiraWitt}{article}{
		title={Frobenius powers},
		author={Hern{\'a}ndez, Daniel J}, author={Teixeira, Pedro}, author={Witt, Emily E},
		journal={Mathematische Zeitschrift},
		volume={296},
		pages={541--572},
		year={2020},
		publisher={Springer}
	}

	\bib{PeskineSzpiro}{article}{
		title={Liaison des vari{\'e}t{\'e}s alg{\'e}briques. I.},
		author={Peskine, Christian}, author={Szpiro, Lucien},
		journal={Inventiones mathematicae},
		volume={26},
		pages={271--302},
		year={1974}
	}


	\bib{TakagiWanatabe}{article}{
		title={On F-pure thresholds},
		author={Takagi, Shunsuke}, author={Watanabe, Kei-ichi},
		journal={Journal of Algebra},
		volume={282},
		number={1},
		pages={278--297},
		year={2004},
		publisher={Elsevier}
	}

\end{thebibliography}
\end{document}